\documentclass[12pt]{amsart}
\usepackage{amssymb, amsfonts, amsthm}
\usepackage{graphicx}

\numberwithin{equation}{section}

\newtheorem{theorem}{Theorem}[section]
\newtheorem{lemma}[theorem]{Lemma}

\theoremstyle{definition}

\newcommand{\A}{{\mathcal A}}

\newcommand{\es}{{\mathcal S}}

\newcommand{\D}{{\mathbb D}}

\newcounter{minutes}
\divide\time by 60
\newcounter{hours}
\multiply\time by 60 \addtocounter{minutes}{-\time}
\begin{document}

\title[$p$-valent starlike functions]
{A note on a class of $p$-valent starlike functions of order beta}

\author{Swadesh Sahoo${}^*$}
\address{Swadesh Sahoo, Discipline of Mathematics,
Indian Institute of Technology Indore,
Indore 452 017, India}
\email{swadesh@iiti.ac.in}

\author{Navneet Lal Sharma}
\address{Navneet Lal Sharma, Discipline of Mathematics,
Indian Institute of Technology Indore,
Indore 452 017, India}
\email{sharma.navneet23@gmail.com}

\thanks{${}^*$ The corresponding author}

\begin{center}
\texttt{File:~\jobname .tex, printed: \number\day-\number\month-\number\year,
\thehours.\ifnum\theminutes<10{0}\fi\theminutes}
\end{center}

\thispagestyle{empty}


\begin{abstract}
In this paper we obtain sharp coefficient bounds for certain $p$-valent starlike functions of order $\beta$, $0\le \beta<1$.
Initially this problem was handled by Aouf in {\em M.K. Aouf, On a class of $p$-valent starlike functions of order $\alpha$,
Internat. J. Math. $\&$ Math. Sci. 1987;10:733--744}. We pointed out that the proof given by Aouf was incorrect and a correct proof is
presented in this paper.   


\end{abstract}

\maketitle
\section{Introduction}
It is well-known that each univalent functions of the form
$$f(z)=z+\sum_{n=2}^\infty a_nz^n
$$
in the open unit disk $\D:= \{z\in \mathbb{C}: |z|<1 \}$ has the property $|a_2|\le 2$, with
equality occurring only for rotations of the Koebe function
$$k(z)=\frac{z}{(1-z)^2}=z+\sum_{n=2}^\infty nz^n.
$$
This suggests the famous conjecture of Bieberbach \cite{Bie16}, first proposed in 1916.
This states that if $f$ in the above form is univalent in $\mathbb{D}$ then $|a_n|\le n$ 
for all $n\ge 2$. Initially this conjecture was proved in many special cases 
and has a long history. It was finally settled after several years by De Branges 
\cite{Bra85} in 1985. For basic theory of Bieberbach
conjecture problem for number of classes of univalent functions we refer to
\cite{Dur83,Goo83}. Part of this development, it was not generalized to the class of $p$-valent functions until 1948. The initiative was first taken by Goodman, see \cite{Goo48}.
Similar problem for many other classes of $p$-valent functions can be found, for instance in \cite{Aouf87,Gol74,PT83}.
In this paper we consider certain classes of $p$-valent functions in the unit disk
and prove Bieberbach's conjecture for these functions. 

For a natural number $p$, let $\A_{p}$ denote the class of functions of the form
\begin{equation}\label{eq1}
f(z)=z^p+\sum_{n=1}^\infty a_{n+p} z^{n+p}
\end{equation}
which are analytic and $p$-valent in the open unit disk.

Let $g(z)$ and $f(z)$ be analytic in $\D$. A function $g(z)$ is called to be subordinate to 
$f(z)$ if there exists an analytic function $\phi(z)$ in $\D$ with $\phi(0)=0$ and 
$|\phi(z)|<1\, (z\in \D)$ such that $g(z)=f(\phi(z))$.
We denote this subordination by $g(z)\prec f(z)$\, (see \cite{MM2000}).

Let $\es _p(A,B,\beta)$ denote the class of functions $f(z)\in \A_p$ satisfying
\begin{equation}\label{eq2}
\frac{zf'(z)}{f(z)} \prec \frac{p+\big[pB+(A-B)(p-\beta)\big]z}{1+Bz}, \quad z\in \D, ~0\le \beta<1,
\end{equation}
where $A$ and $B$ have the restriction $-1\le B<A\le 1.$ The class $\es _p(A,B,\beta)$ was considered by Aouf in \cite{Aouf87}.
As a special case, we see that
$$\es_p(1,-1,\beta)=\es_p(\beta),\, \es_1(\beta)=\es ^*(\beta),\, \es_p(0)=\es_p ~\mbox{and } \es_1(A,B,0)=\es ^*(A,B).
$$
Note that $\es_p(\beta)$, the class of $p$-valent starlike functions of order $\beta$, 
was studied by Goluzina in \cite{Gol74}; $\es ^*(\beta)$, the class of starlike functions of 
order $\beta$ was introduced by Robertson in \cite{Rob36}; $\es_p$, the usual class of $p$-valent starlike functions; and
$\es ^*(A,B)$ was introduced by Janowski in \cite{Jan73}.

Aouf estimated the coefficient bounds for the functions from the class $\es _p(A,B,\beta)$ in \cite{Aouf87} in which 
the proof is found to be incorrect. In this paper, we provide a correct proof.
\section{\bf Main result}
The following Lemma is obtained by Goel and Mehrok:
\begin{lemma}\label{lem1}\cite[Theorem~1]{GM81}
 Let $-1\le B<A\le 1$ and $f\in \es ^*(A,B)$. Then 
\begin{equation}\label{eq1-lem1}
  |a_2|\leq A-B;
\end{equation}
for $A-2B\leq 1,\, n\geq 3,$
\begin{equation}\label{eq2-lem1}
|a_n|\leq \frac{A-B}{n-1};
\end{equation}
and for $A-(n-1)B>(n-2),\, n\geq 3$,
\begin{equation}\label{eq3-lem1}
|a_n|\leq \frac{1}{(n-1)\ !}\prod_{j=2}^{n}(A-(j-1)B).
\end{equation}
The equality signs in $(\ref{eq1-lem1})$ and $(\ref{eq2-lem1})$ are attained for the functions
\begin{equation}\label{eq5-lem1}
k_{n,A,B}(z)= \left \{
\begin{array}{ll}
z(1+B\delta z^{n-1})^{(A-B)/(n-1)B}, & \mbox{ if } B\neq 0;\\[1mm] 
z~\rm{exp} \left(\frac{A\delta z^{n-1}}{n-1}\right), & \mbox{ if } B=0,
\end{array} \quad |\delta|=1,\right.
\end{equation}
and in $(\ref{eq3-lem1})$ equality is attained for the functions
\begin{equation}\label{eq6-lem1}
 k_{A,B}(z)= \left \{
\begin{array}{ll}
z(1+B\delta z)^{(A-B)/B}, & \mbox{ if } B\neq 0;\\[1mm]
ze^{Az\delta}, & \mbox{ if } B=0,
\end{array} \quad |\delta|=1.\right.
\end{equation}
\end{lemma}
However, a $p$-valent analog of Lemma~\ref{lem1} was wrongly proven by Aouf in the following form:

\medskip
\noindent
{\bf Theorem~A.~}\cite[Theorem~3]{Aouf87} 
{\em Let $-1\le B<A\le 1$ and $p\in \mathbb{N}$. If $f(z)=z^p+\sum_{n=p+1}^\infty a_{n} z^{n}\in \es _p(A,B,\beta)$, then
$$|a_n|\leq \prod_{j=0}^{n-p-1}\frac{|(B-A)(p-\beta)+Bj|}{j+1}
$$
for $n\ge p+1$, and these bounds are sharp for all admissible $A,B,\beta$ and for each $n$.
}

We now give the correct form of the statement stated in Theorem~A and it's proof.
\begin{theorem}\label{thm1}
Let $-1\le B<A\le 1$ and $p\in \mathbb{N}$. If $f(z)\in \es _p(A,B,\beta)$ is in the form $(\ref{eq1})$, then we have
\begin{equation}\label{eq1-thm1}
  |a_{p+1}|\leq (A-B)(p-\beta);
\end{equation}
for $A(p-\beta)-B(p-\beta-1)\le 1$ (or $A(p-\beta)-B(n-\beta-1)\le (n-p-1)),\, n\geq p+2,$
\begin{equation}\label{eq2-thm1}
|a_n|\leq \frac{(A-B)(p-\beta)}{n-p};
\end{equation}
and for $A(p-\beta)-B(n-\beta-1)>(n-p-1),\, n\geq p+2$,
\begin{equation}\label{eq3-thm1}
|a_n|\leq \prod_{j=1}^{n-p}\frac{(A(p-\beta)-B(p-\beta+j-1))}{j}.
\end{equation}
The inequalities $(\ref{eq1-thm1}), (\ref{eq2-thm1})$ and $(\ref{eq3-thm1})$ are sharp.

\end{theorem}
\begin{proof}
Let $f(z)\in \es _p(A,B,\beta)$. By the relation~(\ref{eq2}) we can guarantee 
an analytic function $\phi:\D\rightarrow \overline{\D}$ with $\phi(0)=0$ such that
$$\frac{zf'(z)}{f(z)}=\frac{p+\big[pB+(A-B)(p-\beta)\big]\phi(z)}{1+B\phi(z)},
$$
i.e.
$$zf'(z)-pf(z)=\big[(pB+(A-B)(p-\beta))f(z)-Bzf'(z)\big]\phi(z).
$$
Substituting the series expansion~(\ref{eq1}), of $f(z)$, and canceling the factor $z^p$ on both sides, we obtain
$$\sum_{k=1}^{\infty}ka_{p+k}z^k=\left((A-B)(p-\beta)-\sum_{k=1}^{\infty}\big[B(p+k)+
(-pB+(B-A)(p-\beta))\big]a_{p+k}z^k\right)\phi(z).
$$
Rewriting it, we get
$$\sum_{k=1}^{\infty}ka_{p+k}z^k=\left((A-B)(p-\beta)+\sum_{k=1}^{\infty}\big[A(p-\beta)-B(k+p-\beta)\big]a_{p+k}z^k\right)\phi(z).
$$
By Clunie's method~\cite{Clu59}~(for instance see~\cite{Rog43,Rob70}) for $n\in \mathbb{N}$, we observe that
$$\sum_{k=1}^{n}k^2|a_{p+k}|^2\le (A-B)^2(p-\beta)^2+\sum_{k=1}^{n-1}\big[A(p-\beta)-B(k+p-\beta)\big]^2|a_{p+k}|^2.
$$
Simplification of the above inequality leads to
$$|a_{p+n}|^2\le \frac{1}{n^2}\left((A-B)^2(p-\beta)^2+\sum_{k=1}^{n-1}\Big(\big[A(p-\beta)
-B(k+p-\beta)\big]^2-k^2\Big)|a_{p+k}|^2\right)
$$
or
\begin{align*}
&|a_{p+n}|^2\le \frac{1}{n^2}\Bigg((A-B)^2(p-\beta)^2+\sum_{k=2}^{n}\Big(\big[A(p-\beta)
-B(k+p-\beta-1)\big]^2\\
&\hspace{9cm}-(k-1)^2\Big)|a_{p+k-1}|^2\Bigg).
\end{align*}
Above inequality can be rewritten by replacing $p+n$ by $n$ as
\begin{align}\label{eq7-thm1}
&|a_{n}|^2\le \frac{1}{(n-p)^2}\Bigg((A-B)^2(p-\beta)^2+\sum_{k=2}^{n-p}\Big(\big[A(p-\beta)
-B(k+p-\beta-1)\big]^2\\ \nonumber
&\hspace{9cm}-(k-1)^2\Big)|a_{p+k-1}|^2\Bigg)
\end{align}
for $n\ge p+1$.\\
Note that the terms under the summation in the right hand side of (\ref{eq7-thm1}) may be positive as well as negative.
We investigate it by including here a table (see Table~1) for values of $W:=\big(A(p-\beta)-B(k+p-\beta-1)\big)^2-(k-1)^2$ for various
choices of $A,B,k,\beta$ and $p$.
\begin{center}\
\begin{tabular}{|c|l|l|l|l|l|}
\hline
 \textbf{k} & $p$ & \textbf{A} &  \textbf{ B} & \textbf{$\beta$} & \textbf{W} \\
\hline
2&1&0.8  &0.5&0& -0.96 \\                                                   
\hline
2&1&-0.5  &-0.8& 0 & 0.21 \\ 
\hline
3&2&0.5  &0.4& 0.5 & -3.5775 \\ 
\hline
3&2&-0.1 &-0.7&0.5& 1.29 \\ 
\hline
\end{tabular}
\end{center}
\hspace{7.5cm} Table~1
\begin{center}
(\it This the place where the incorrectness of Aouf's proof is found!)
\end{center}
So, we can not apply direct mathematical induction in (\ref{eq7-thm1}) to establish the required bounds for $|a_n|$.
Therefore, we are considering different cases for this.

First, for $n=p+1$, we easily see that (\ref{eq7-thm1}) reduces to
$$|a_{p+1}|\le (A-B)(p-\beta)
$$
which establishes~(\ref{eq1-thm1}).

Secondly, $A(p-\beta)-B(p-\beta-1)\le 1$ if and only if $A(p-\beta)-B(n-\beta-1)\le (n-p-1)$ for $n\ge p+2$. Since all 
the terms under the summation in (\ref{eq7-thm1}) are non-positive, we reduce to
$$|a_n|\leq \frac{(A-B)(p-\beta)}{n-p}
$$
for $A(p-\beta)-B(p-\beta +1)\leq 1,\, n\geq p+2.$ This proves~(\ref{eq2-thm1}).
The equality holds in $(\ref{eq1-thm1})$ and $(\ref{eq2-thm1})$ for the functions
$$
k_{n,A,B,p}(z)= \left \{
\begin{array}{ll}
z^p\big(1+B\delta z^{n-1}\big)^{(A-B)(p-\beta)/(n-1)B}, &  B\neq 0;\\[1mm]
z^p~\rm{exp} \left(\frac{A(p-\beta)\delta z^{n-1}}{n-1}\right), &  B=0,
\end{array}\quad |\delta|=1.\right.
$$

Finally let us prove (\ref{eq3-thm1}) when $A(p-\beta)-B(n-\beta-1)>(n-p-1),\, n\geq p+2.$
We see that all the terms under the summation in (\ref{eq7-thm1}) are positive. We prove the inequality by the 
usual mathematical induction.
Fix $n,$ $n\ge p+2$ and suppose that (\ref{eq3-thm1}) holds for $k=3,4,\ldots,n-p.$ Then from (\ref{eq7-thm1}),
we find
\begin{align}\label{eq8-thm1}
&|a_{n}|^2\le \frac{1}{(n-p)^2}\left((A-B)^2(p-\beta)^2+\sum_{k=2}^{n-p}\Big(\big[A(p-\beta)
-B(k+p-\beta-1)\big]^2-(k-1)^2\Big)\right.\\\nonumber
&\hspace{8cm}\left. \prod_{j=1}^{k-1}\frac{\big[A(p-\beta)-B(p-\beta+j-1)\big]^2}{j^2}\right).
\end{align}
It is now enough to show that the square of the right hand side of (\ref{eq3-thm1}) is equal to the 
right hand side of (\ref{eq8-thm1}), that is
\begin{align}\label{eq9-thm1}
&\prod_{j=1}^{m-p}\frac{\big[A(p-\beta)-B(p-\beta+j-1)\big]^2}{j^2}=\frac{1}{(m-p)^2}\Bigg((A-B)^2(p-\beta)^2\\ \nonumber
& +\sum_{k=2}^{m-p}\Big(\big[A(p-\beta)-B(k+p-\beta-1)\big]^2-(k-1)^2\Big)
 \prod_{j=1}^{k-1}\frac{\big[A(p-\beta)-B(p-\beta+j-1)\big]^2}{j^2}\Bigg)
\end{align}
for $A(p-\beta)-B(m-\beta-1)>(m-p-1),\, m\geq p+2.$ We also use the induction principle to prove (\ref{eq9-thm1}).

The equation~(\ref{eq9-thm1}) is recognized for $m=p+2$. Suppose that (\ref{eq9-thm1}) is true for all $m,\, p+2<m\le n-p$.
Then from (\ref{eq8-thm1}), we obtain
\begin{align*}
&|a_{n}|^2\le \frac{1}{(n-p)^2}\Bigg((A-B)^2(p-\beta)^2+\sum_{k=2}^{n-p-1}\Big(\big[A(p-\beta)
-B(k+p-\beta-1)\big]^2-(k-1)^2\Big)\\\nonumber
&\hspace{3cm}\times \prod_{j=1}^{k-1}\frac{\big[A(p-\beta)-B(p-\beta+j-1)\big]^2}{j^2}
+\Big(\big[A(p-\beta)-B(n-\beta-1)\big]^2\\
& \hspace{4cm}	-(n-p-1)^2\Big)\times \prod_{j=1}^{n-p-1}\frac{\big[A(p-\beta)-B(p-\beta+j-1)\big]^2}{j^2}\Bigg).
\end{align*}
Using the induction hypothesis, for $m=n-1,$ we get
\begin{align*}
&|a_{n}|^2\le \frac{1}{(n-p)^2}\left( (n-p-1)^2\prod_{j=1}^{n-p-1}\frac{\big[A(p-\beta)-B(p-\beta+j-1)\big]^2}{j^2}\right.\\
&\hspace{.3cm}\left.+\Big(\big[A(p-\beta)-B(n-\beta-1)\big]^2-(n-p-1)^2\Big)
\prod_{j=1}^{n-p-1}\frac{\big[A(p-\beta)-B(p-\beta+j-1)\big]^2}{j^2}\right).
\end{align*}
Hence
$$|a_{n}|\le \prod_{j=1}^{n-p}\frac{\big[A(p-\beta)-B(p-\beta+j-1)\big]}{j}.
$$
It is easy to prove that the bounds are sharp for the function
$$k_{A,B,p}(z)= \left \{
\begin{array}{ll}
z^p\big(1+B\delta z\big)^{(A-B)(p-\beta)/B}, &  B\neq 0;\\[1mm]
z^pe^{A(p-\beta)z\delta}, &  B=0,
\end{array}\quad |\delta|=1.\right.
$$
This completes the proof of Theorem~\ref{thm1}.
\end{proof} 
We remark that, choosing $p=1$ and $\beta=0$ in Theorem~\ref{thm1} we turned into Lemma~\ref{lem1}.

\bigskip
\noindent
{\bf Acknowledgements.} 
The second author acknowledges the support of National Board for Higher Mathematics, Department of Atomic Energy,
India (grant no. 2/39(20)/2010-R$\&$D-II).

\end{document}